\documentclass[11pt]{article}

	\usepackage[utf8]{inputenc}
	\usepackage[T1]{fontenc}
	\usepackage[english]{babel}
	\usepackage{amsmath, amsthm, amssymb}
	\usepackage{cite, enumerate}

\theoremstyle{plain}
	\newtheorem{lemma}{Lemma}
	\newtheorem{proposition}{Proposition}
	\newtheorem{theorem}{Theorem}
	
\theoremstyle{definition}
	\newtheorem{definition}{Definition}
\theoremstyle{remark}
	\newtheorem{remark}{Remark}
	\newtheorem{example}{Example}
    
\newcommand{\fn}{\mathfrak{n}}
\newcommand{\fa}{\mathfrak{a}}
\newcommand{\fh}{\mathfrak{h}}
\newcommand{\fd}{\mathfrak{d}}
\newcommand{\fg}{\mathfrak{g}}
\newcommand{\fb}{\mathfrak{b}}
\newcommand{\fgl}{\mathfrak{gl}}
\newcommand{\fsl}{\mathfrak{sl}}
\newcommand{\fs}{\mathfrak{s}}
\newcommand{\fr}{\mathfrak{r}}
\newcommand{\fso}{\mathfrak{so}}
\newcommand{\fsu}{\mathfrak{su}}

\newcommand{\ku}{\mathbb{K}}
\newcommand\blank{{\mkern 2mu\cdot\mkern 2mu}}

\DeclareMathOperator{\ad}{ad}

\DeclareMathOperator{\der}{Der}
\DeclareMathOperator{\determ}{det}
 
\DeclareMathOperator{\spa}{span} 
\DeclareMathOperator{\im}{Im} 
\DeclareMathOperator{\inner}{Inner}

\DeclareMathOperator{\linear}{End}
\DeclareMathOperator{\tr}{tr}

\title{Metrics related to Oscillator algebras}
\author{Pilar Benito and Jorge Rold\'an-L\'opez}
\date{December 23, 2022}

\begin{document}
\maketitle

\begin{abstract}
    A Lie algebra is said to be metric if it admits a symmetric invariant and nondegenerate bilinear form. The harmonic oscillator algebra, which arises in the quantum mechanical description of a harmonic oscillator, is the smallest solvable nonabelian metric example. This algebra is the first step of a countable series of solvable Lie algebras which support invariant Lorentzian forms.
    Generalizing this situation, in this paper we arrive to the oscillator Lie $\ku$-algebras as double extensions of metric spaces. The aim of this paper is to present some structural features, invariant metrics and derivations of this class of algebras and to explore their possibilities of being extended to mixed metric Lie algebras.
\end{abstract}

 \medskip

    
\section{Introduction}

Real oscillator algebras are the Lie algebras attached to connected, simply connected and non-simple Lie groups that admit a Lorentz invariant metric that makes them indecomposable (see~\cite[Theorem 4.1]{medina1985groupes}). This type of algebras were introduced in \cite{hilgert1985lorentzian} as double extensions of Hilbert spaces, and renamed as standard solvable \emph{Lorentzian Lie algebras $A_{2m+2}$} in~\cite[Definition II.3.16]{hilgert1989lie}. Also, they are the class of real solvable non-abelian Lie algebras that carry an invariant inner product of metric signature $(2m+1,1)$. This class is integrated by real Lie algebras of dimension $2m+2$ for $m\geq 1$. The term oscillator comes from quantum mechanics because these algebras describe a system of harmonic oscillator $m$-dimensional euclidean space. The $4$-dimensional standard Lorentzian algebra $A_4$ (in this paper denoted by  $\fd_4(\mathbb{R})$) is the \emph{harmonic oscillator algebra}, and it is the Lie algebra of the harmonic oscillator group (see \cite[Example V.4.15]{hilgert1989lie} and \cite{douglas2007class}). Oscillator algebras also support other nonassociative structures such as Poisson and Leibniz algebras and symmetric Leibniz bialgebras following \cite{camacho2019leibniz} and \cite{albuquerque2021poisson}.

A \emph{Lie algebra} $\fg$ is a vector space over a field $\mathbb{K}$ endowed with a binary skew-symmetric ($\frac{1}{2}\in \mathbb{K}$) bilinear product $[x,y]$ satisfying the \emph{Jacobi identity}:
\begin{equation}\label{eq:Jacobi}
J(x,y,z)=[[x,y],z]+[[z,x],y]+[[y,z],x]=0\quad \forall\, x,y,z\in \fg.
\end{equation}
In case $[x,y]=0$ for every $x,y\in \fg$, the Lie algebra $\fg$ is called \emph{abelian} and identity~\eqref{eq:Jacobi} becomes trivial. Easy-to-follow examples of Lie algebras are the set of linear maps over a finite-dimensional vector space $V$ (equivalently the set of $n\times n$ matrices) under the commutator bracket $[f,g]=fg-gf$ ($[A,B]=AB-BA$ in a matrix level). This algebra is called \emph{general Lie algebra} and will be denote along this paper as $\fgl(V)$ or $\fgl_n(\ku)$ in its matrix form. Any subalgebra of $\fgl(V)$ is named a \emph{linear Lie algebra}.

The global structure of invariant Lie groups is encoded in the algebraic structure of their (real) metric Lie algebras. In 1985, Medina (see~\cite[Lemma 2.1 and Corollary 2.2]{medina1985groupes}) provides the following equivalent conditions on the existence of bi-invariant metrics on Lie groups:

\begin{quote}
   \emph{ For a given Lie group $G$ and its Lie algebra $Lie(G)=\fg$, the following statements are equivalents,
    \begin{itemize}
        \item [(a)] The group $G$ is endowed with a bi-invariant metric.
        \item [(b)] The algebra $\fg$ has a metric such that the adjoint action of $G$ on $\fg$ is given by isometries.
        \item [(c)] The adjoint and coadjoint representations of $\fg$ are isomorphic by means of an isomorphism $\psi: g \to g^*$ that satisfies $\psi(a)(b) = \psi(b)(a)$.
    \end{itemize}
    Moreover, if $G$ is a connected group, (a) is also equivalent to,
    \begin{itemize}
        \item [(d)] The algebra $\fg$ has a quadratic form $q\colon \fg \to \mathbb{R}$ and for every $x\in \fg$, the linear transformations $\ad x$ is skew-adjoint with respect to the bilinear $\varphi$ form attached to $q$, i.e. $\varphi(x, y)=q(x+y)-q(x)-q(y)$.
    \end{itemize}
    }
\end{quote}

\noindent These conditions have been previously established by Milnor in 1976 (see~\cite[Lemmas 7.1 and 7.2]{milnor1976curvatures}). Since $\ad x(y)=[x,y]$ is the left product on $\fg$, condition (d) is equivalent to
\begin{equation}\label{eq:invariant-metric}
    \varphi([x,y],z)+\varphi( y,[x,z])=0,
\end{equation}
for all $x,y,z\in \fg$.
Any arbitrary Lie algebra $\fg$ admitting a nondegenerate symmetric bilinear form $\varphi$ that satisfies equation~\eqref{eq:invariant-metric} is named \emph{metric Lie algebra}, and $\varphi$ is called \emph{invariant symmetric form} on $\fg$. This notion appears for the first time in \cite{Tsou_Walker_1957} under the name of \emph{metrisable algebra}. These algebras are also known as quadratic, orthogonal, metric or metrised (these terms usually used over the reals) or self-dual. Along this paper, taking into account that oscillator algebras arise over the real field and following \cite{bordemann1997nondegenerate}, any arbitrary Lie algebra $\fg$ admitting a nondegenerate invariant bilinear form $\varphi$ will be named \emph{pseudo-metric} and if in addition $\varphi$ is symmetric, we will say that $\fg$ is a \emph{metric algebra}. In \cite[Lemma 1]{hilgert1996orthogonal} it is proved that for real Lie algebras pseudo-metric and metric notions are equivalent. Metric Lie algebras that decompose as the orthogonal sum of two ideals are called \emph{decomposable}. Otherwise, they are called \emph{indecomposable}.

A simple Lie algebra is a non-abelian Lie algebra without proper ideals different form zero. Simple Lie algebras are the brick blocks that define the class of semisimple Lie algebras. As a main example of a simple Lie algebra, we point out $\fsl(V)$ or $\fsl_{m+1}(\ku)$ in matrix form, the set of traceless linear maps of $V$ or traceless matrices. The algebra $\fsl(V)$ is a subalgebra of $\fgl(V)$, called \emph{special linear algebra} due to its connection with the special linear Lie group $SL(V)$ of endomorphisms of determinant $1$.

Semisimple Lie algebras under the Killing form, $\kappa(x,y)=\tr(\ad x\ad y)$, are pretty examples of metric algebras. In the opposite structural side, we find abelian Lie algebras, all of them are metric by using any non-degenerate symmetric form. In \cite{medina1985algebres} it is shown that every metric Lie algebra which is neither simple nor one-dimensional is a double extension of a metric Lie algebra $(\fg, \varphi)$ by a simple Lie algebra or a one-dimensional Lie algebra.

 Oscillator algebras are metric and solvable Lie algebras with nilradical a Lie algebra of Heisenberg type according to \cite[Proposition II.3.11]{hilgert1989lie}. Also, they are local algebras (only one maximal ideal), and, therefore, they are indecomposable as metric Lie algebras. The main goal of this paper is to generalize the notion of oscillator algebras to arbitrary fields of characteristic zero and, as metric algebras they are, to study their double extensions to other mixed metric Lie algebras. In Section 2 we introduced general oscillator $\ku$-algebras in Definition \ref{def:general-oscillator} and we check in Proposition \ref{prop:pattern-oscillator} that they preserve the structural basic properties that they have as $\mathbb{R}$-algebras. Our Proposition \ref{prop:inv-sym-mdim} expand to arbitrary fields Lemma 1 in \cite{hilgert1996orthogonal}, and Proposition \ref{prop:der-oscillator} relates derivations of oscillator algebras to derivations of Heisenberg algebras. Section 3 is devoted to compute the vector space of invariant forms of this type of algebras (Lemma 1) and to give a explicit matrix description of the whole set of derivations and skew-derivations of the class of some basic oscillator $\mathbb{R}$-algebras (Theorem 2). The description points out that, for $m\geq 2$ the algebra of skew-derivations is a mixed Lie algebra with Levi subalgebra the special unitary real Lie algebra $\fsu_m(\mathbb{R})$. So, from oscillator algebras we can get mixed metric Lie algebras. Double extensions of the harmonic oscillator $\fd_4(\mathbb{R})$ only produces decomposable metric solvable Lie algebras. But this little algebra, also known as \emph{diamond} algebra, has a self-interest of its own (see \cite{douglas2007class} and \cite{casati2010indecomposable} and references therein).

\section{Metric spaces and Oscillator algebras}

In the sequel, $\fg$ will denote a Lie algebra with product $[x,y]$ over a field $\ku$ of characteristic zero and $\varphi: \fg \times \fg \to \ku$ an invariant and non-degenerate bilinear form. So $[x,y]=-[y,x]$ and equations  \eqref{eq:Jacobi} and \eqref{eq:invariant-metric} are fulfilled. Any self-linear map of $\fg$ that satisfies $d[x,y]=[d(x), y]+[x,d(y)]$ for all $x,y\in \fg$ is called \emph{derivation}. Left products $\ad x=[x,\cdot]$ are examples of derivations which are named \emph{inner derivations}. A linear map $d$ is said \emph{skew} with respect to $\varphi$ or $\varphi$-skew if
\begin{equation}\label{skew-der}
    \varphi(d(x), y)+\varphi(x,d(y))=0 \quad \text{ for all } x,y \in \fg.
\end{equation}
We will denote $\der \fg$, $\inner \fg$ and $\der_\varphi \fg$ the sets of derivations, inner derivations and $\varphi$-skew derivations of the metric Lie algebra $(\fg, \varphi)$. Those sets are subalgebras of the general linear algebra $\fgl(\fg)$. For a vector space $V$, the set of $1$-forms $\beta\colon V\to \ku$, will be denote by $V^*$ (\emph{dual space} of $V$). The derived series (lower central series) of $\fg$ is defined recursively as $\fg^{(1)}=\fg$ (respectively $\fg^{1}=\fg$) and $\fg^{(n+1)}=[\fg^{(n)},\fg^{(n)}]$ (respectively $\fg^{n+1}=[\fg^n,\fg^n]$). A Lie algebra in which the derived (lower central) series reaches zero is named solvable (nilpotent). Semisimple Lie algebras are direct sum as ideals of simple ones. According to Levi Theorem, any Lie algebra decomposes as a direct sum (as subalgebras) of its (unique) maximal solvable ideal, $\mathfrak{r}(\fg)$ and a semisimple subalgebra $\fs$ called \emph{Levi subalgebra}. So $\fg=\fs\oplus \mathfrak{r}(\fg)$ and, along this paper, $\fg$ is said \emph{mixed Lie algebra} if both summands are non zero. The maximal nilpotent ideal of $\fg$ will be denote as $\fn(\fg)$ and named \emph{nilradical} of $\fg$.

According to \cite{medina1985algebres}, from a metric Lie algebra and through out the double extension process, we can build new metric Lie algebras. Here we use a more general presentation of this method that can be found in \cite{bordemann1997nondegenerate}.

\begin{theorem}\label{thm:doubleExtension}
Let $(\fg, \varphi)$ be a finite-dimensional metric Lie algebra over a field $\mathbb{K}$. Let $\fb$ be another finite-dimensional Lie algebra over $\mathbb{K}$ and suppose there is a Lie homomorphism $\phi\colon \fb \to \der_\varphi(\fg)$. Denote by $w\colon \fg\times \fg \to \fb^*$ the bilinear skew-symmetric map $(a, a') \mapsto (b \mapsto \varphi(\phi(b)(a),a'))$. Take the vector space direct sum $\fg_\fb:= \fb \oplus \fg \oplus \fb^*$ and define the following multiplication for $b$, $b' \in \fb$, $a$, $a'\in \fg$, and $\beta, \beta'\in \fb^*$:
	\begin{multline}\label{eq:de-bracket}
		[b + a + \beta, b' + a' + \beta'] :=  [b,b']_\fb + \phi(b)(a') - \phi(b')(a) + [a,a']_\fg \\+ w(a, a') + \beta'\circ \ad b  - \beta\circ \ad b'.
	\end{multline}
	Moreover, define the following symmetric bilinear form $\varphi_\fb$ on $\fg_\fb$:
	\begin{equation}\label{eq:de-bilinear}
		\varphi_\fb(b + a + \beta, b' + a' + \beta') := \beta(b') + \beta'(b) + \varphi(a, a').
	\end{equation}
Then the pair $(\fg_\fb, \varphi_\fb)$ is a metric Lie algebra over $\mathbb{K}$ and is called the double extension of $(\fg, \varphi)$ by $(\fb, \phi)$. \hfill $\square$
\end{theorem}

\begin{example} [One-dimensional double extension]

    Starting with the $\ku$-vector space $V_2=\spa \langle x,y \rangle$ and the bilinear symmetric and nondegenerate form $\varphi(x,x)=\varphi(y,y)=1$ and $\varphi(x,y)=0$, the set $\linear_\varphi (V_2)$ of linear maps that satisfy expression~\eqref{skew-der} is a one-dimensional subspace generated by the linear isomorphism $\delta(x)=y$ and $\delta(y)=-x$. Then, the Lie algebra obtained as double extension of the abelian metric algebra $(V_2,\varphi)$ by $\fb=\ku\cdot \delta=\linear_\varphi (V_2)=\der_\varphi V_2$ is the $4$-dimensional Lie algebra $\fd_4(\ku)=\ku\cdot \delta \oplus V_2\oplus \ku\cdot \delta^*$. Where the Lie bracket given in expression~\eqref{eq:de-bracket} turns into (in this case $\beta\circ \ad b=0$ and $w(a,a')=\varphi(\delta(a),a')\delta^*$):
    \begin{multline}\label{eq:bracket-oscillator}
    [t\delta + \lambda_1 x+ \lambda_2 y+ s\delta^*,t'\delta + \mu_1 x+ \mu_2 y+ s'\delta^*] := \\ (t\mu_1-t'\lambda_1)y - (t\mu_2-t'\lambda_2)x + (\lambda_1\mu_2- \lambda_2\mu_1)\delta^*,
	\end{multline}
for $t,t',s,s' \in \ku$ (for basic generators, $[\delta,x]=y, [\delta,y]=-x$ and $[x,y]=\delta^*$). The bilinear form extends $\varphi$ to $\varphi_{\fb}$ by declaring the orthogonal decomposition $\fd_4(\ku)=\spa \langle \delta,\delta^* \rangle \perp V_2$ where $\delta$ and $\delta^*$ are isotropic vectors and $\varphi_\fb(\delta,\delta^*)=1$. So $W=\spa \langle \delta,\delta^* \rangle$ is an hyperbolic subspace. In this way we arrive to a $4$-dimensional algebra which is the smallest nonabelian solvable metric Lie $\ku$-algebra.  
The real algebra $\fd_4(\mathbb{R})$ is the harmonic oscillator algebra, with metric signature $(3,1)$. This algebra is also known as diamond or Nappi-Witten Lie algebra (see \cite{casati2010indecomposable} and references therein) and can be obtained as the central extension of the Poincaré Lie algebra in two dimensions.
\end{example}
\begin{definition}\label{def:general-oscillator}
    Let $(V_n, \varphi)$ be a $n$-dimensional $\ku$-vector space endowed with a symmetric and non-degenerate form $\varphi$. Consider now any skew-linear automorphism $\delta$ of $V_n$; it is only possible for $n=2m$. The double extension $\fd(V_{2m}, \varphi, \delta)=\ku \cdot \delta\oplus V_{2m}\oplus \ku \cdot \delta^*$ where $\delta ^*$ is the dual $1$-form of $\delta$ is a metric Lie algebra that we will call generalized oscillator $\ku$-algebra on the triple $(V_{2m},\varphi, \delta)$ (to shorten $\fd_{2m+2}(\ku)$). 
\end{definition}

\begin{remark}
    The algebras $\fd_{2m+2}(\mathbb{R})$ are the real oscillator algebras introduced firstly in \cite[Proposition 2.2]{hilgert1985lorentzian} as Lorentzian semialgebras of Class II (see also \cite[Proposition II.3.11]{hilgert1989lie}). In \cite[Definition II.7]{neeb1993invariant} and \cite{hilgert1996orthogonal}  they appear as remarkable examples of Lie algebras with cone potencial. The study of other nonassociative structures on oscillator algebras in \cite[Section 5]{albuquerque2021poisson} yields to some geometric information on connections and metrics on oscillator Lie groups. 
\end{remark}

A representation of $\fg$ is a homomorphism of Lie algebras $\rho\colon \fg \to \fgl(V)$, so $\rho([x,y])=\rho(x)\rho(y)-\rho(y)\rho(x)$. Here the vector space $V$ is called $\fg$-module. As main examples of modules we have the adjoint $\rho=\ad$ and coadjoint $\rho = \ad^*$ representations. The first one provides $V=\fg$ as $\fg$-module ($\ad x (y)=[x, y]$) and the last one the dual module $V^*=\fg^*$ (here $\ad^* x(\beta)=-\beta\circ \ad x$). The vector space of $\fg$-module homomorphisms, denoted as $Hom_\fg (\fg, \fg^*)$, consists on the linear maps $f\colon \fg \to \fg^*$ such that
\begin{equation}\label{eq:map-module}
    f([x,y])=-f(y)\circ \ad x \text{ for all  }x,y\in \fg.
\end{equation}
The algebra $\fg$ is called \emph{self-dual} if the adjoint and coadjoint representations are isomorphic (there is a bijective $\fg$-module homomorphism from $\fg$ to $\fg^*$).

We come back now to items (c) and (d) of the multiple characterisation result of Lie groups with bi-invariant metrics and their Lie algebras (see~\cite{medina1985groupes}) that we have mentioned in the introductory section. Both items highlight the natural relationship between invariant forms on Lie algebras and homomorphisms of their adjoint and coadjoint representations. As vector spaces, the set of invariant bilinear forms, $B_{inv}(\fg)$ and $Hom_\fg (\fg, \fg^*)$, are isomorphic:
\begin{equation}\label{eq:isom-inv-isom-Lmod}
    \Delta\colon B_{inv}(\fg) \to Hom_\fg (\fg, \fg^*), \quad \Delta(b)(x)=b(x,\blank)=\psi_b(x),
\end{equation}
 and $\Delta^{-1}(\psi)=b_\psi$, $b_\psi(x,y)= \psi(x)(y)$. Even more, $\Delta$ sends a non-degenerate invariant form into a $\fg$-module isomorphism and conversely. 

On the other hand, for any bilinear form $\varphi$ of $\fg$, we can set the bilinear form $\varphi^t(x,y):=\varphi(y,x)$. From the anticommutativity of $\fg$ it is easily checked that $\varphi^t$ is invariant if and only if $\varphi$ so is. In this way, over fields of characteristic not 2, the usual decomposition of $\varphi$ as sum of its symmetric part $\varphi_s$ and its skew-symmetric part $\varphi_{as}$
\begin{equation*}
     \varphi=\frac{1}{2}(\varphi+\varphi^t)+\frac{1}{2}(\varphi-\varphi^t)
\end{equation*}
preserves the invariance, and, therefore, $B_{inv}(\fg)$ decomposes as the direct sum of the vector spaces of symmetric invariant forms, $B_{inv}^s$ and that of the skew-symmetric forms $B_{inv}^{as}$ :
\begin{equation*}
    B_{inv}(\fg)=B_{inv}^s(\fg)\oplus B_{inv}^{as}(\fg).
\end{equation*}
So for invarint forms, at a matrix level, we recover the natural decomposition of a matrix as sum of a symmetric matrix and a skew-symmetric matrix.

\begin{definition}
    Let $\fg$ a Lie $\ku$-algebra, the metric dimension $m(\fg)$ of $\fg$ is the dimension of the vector space $B_{inv}(\fg)$. 
\end{definition}

Our next result restates and expands Lemma~1 in \cite{hilgert1996orthogonal} and it clarifies the equivalent assertions for Lie algebras attached to Lie gropus with bi-invariant metrics given by Medina and Milnor.

\begin{proposition}\label{prop:inv-sym-mdim}
    Let $\fg$ be a Lie algebra over a field of characteristic different from $2$. Then it is equivalent:
    \begin{itemize}
        \item [(a)] There exists a nondegenerate form $\varphi\in B_{inv}(\fg)$.
        \item [(b)] There exists a nondegenerate form $\varphi\in B_{inv}^s(\fg)$. 
        \item [(c)] The adjoint and coadjoint representations of $\fg$ are isomorphic.
    \end{itemize}
    Moreover, over fields that have at least $\dim \fg +1$ elements, the vector space $B_{inv}(\fg)$ is linearly generated by the set of invariant and non-degenerate symmetric bilinear forms and previous assertions are also equivalent to:
    \begin{itemize}
        \item [(d)] $m(\fg)$ is greater or equal than one.
    \end{itemize}
\end{proposition}

\begin{proof}
    Let us assume (a) and, using \cite[Proposition 2.4]{bordemann1997nondegenerate}, (b) follows. From (b) we get (c) taking into account that the isomorphism $\Delta$ in expression~\eqref{eq:isom-inv-isom-Lmod} sends any non-degenerate invariant form $b$ into the isomorphism $\varphi_b\colon\fg\to\fg^*$, $\varphi_b(x)=b(x,y)$. This map is one-to-one because $b$ is nondegenerate and, since $\fg$ and $\fg^*$ are equidimensionals, it is bijective. Moreover as $b$ is invariant, for all $x,y \in \fg$:
    \begin{equation*}
        \varphi_b([x,y])=b([x,y], \cdot)=-b(y, \ad x(\cdot))=-b(y,\cdot)\circ \ad x=-\varphi_b(y) \circ \ad x.
    \end{equation*}
    This is just equation~\eqref{eq:map-module}, so $\varphi_b\in Hom_\fg(\fg,\fg^*)$ and adjoint and coadjoint representations are isomorphic. Finally, for any bijective map $\psi \in Hom_\fg (\fg, \fg^*)$, $b_\psi(x,y)=\psi(x)(y)$ is a bilinear form and $b_\psi(x,\fg)=0$ implies that $\psi(x)$ is a null map, so $x=0$ because $\psi$ is one-to-one. Thus $b_\psi$ is nondegenerate. Moreover  $b_\psi([x,y], z)=b_\psi(\ad x(y), z)=\psi([x,y])(z)$ and from equation~\eqref{eq:map-module}, $\psi([x,y])=-\psi(y)\circ \ad x$ and  $\psi([x,y])(z)=-\psi(y)([x,z])=-b_\psi(y, [x,z])$. Hence $b_\psi \in B_{inv}(\fg)$, and (a) follows. 
    According to Lemma 2.1 in \cite{bajo1997lie}, $B_{inv}(\fg)=\spa \langle B_{inv}^s(\fg)\rangle$ over the real field. The arguments in the proof of this lemma are also valid for fields that have at least $\dim \fg+1$ elements. And the equivalence of the four statements is then proved.
\end{proof}

We point out a final pattern on metric Lie algebras (see \cite[Corollary 1.4]{hofmann1986invariant}). Here $U^\perp$ is the orthogonal subspace of any subspace of a vector space $(V,\varphi)$ where $\varphi\colon V\times V\to \ku$ is a symmetric bilinear form.

\begin{proposition}\label{prop:ideals}
    In any Lie algebra $\fg$ endowed with an non-degenerate symmetric invariant form, the map $I\mapsto I^\perp$ is an involutive anti-automorphism of the lattice of ideals of $\fg$ that maps $\fg^2=[\fg,\fg]$ to the centre $Z(\fg)$ and, more generally, the descending central series to the ascending central series. This map also sends maximal ideals into minimal and vice-versa.
\end{proposition}
From Definition \ref{def:general-oscillator} and Theorem \ref{thm:doubleExtension}, the $\ku$-algebras $\fd(V_{2m},\varphi, \delta)$ are one-dimensional double extensions of a metric abelian algebra by non-singular self-derivations with bracket product (see expressions~\eqref{eq:de-bracket} and~\eqref{eq:bracket-oscillator}),
\begin{equation}\label{eq:oscillator-bracket}
    [t\delta+u+s\delta^*,t'\delta+v+s'\delta^*]_{\fd_{2m+2}}=t\delta(v)-t'\delta(u)+\varphi(\delta(u), v)\delta^*.
\end{equation}
The metric structure (there may be other) in $\fd(V,\varphi, \delta)$ is given by expression~\eqref{eq:de-bilinear}.

Following \cite{dixmier1996enveloping}, a $\ku$-algebra of \emph{Heisenberg type} is a Lie algebra $\fh$ whose center is one-dimensional and equal to its derived algebra. For such algebra, the Lie bracket, $[x,y]=b_z(x,y)z$ (here $z$ is the only central element, up to scalars), provides an alternanting bilinear form $b_z:\fh\times \fh \to \ku$ and $(\fh)^\perp=\fh^2=\ku\cdot z$. Note that $b_z$ is non-degenerate on any complement summand $V$ of the centre $Z(\fh)$ in $\fh$. Hence $V$ is a vector space of dimension even and $b_z|_{V\times V}$ has a canonical basis $\{u_1, \dots, u_m,u_{m+1}, \dots, u_{2m}\}$ such that $b_z(u_i,u_{m+i})=1=-b_z(u_{m+i},u_i)$. Then, Heisenberg algebras have odd dimension and, for any natural $m\geq 1$, there is a unique Lie algebra of Heisenberg type of dimension  $2m+1$ described by the \emph{standard basis} given in expression~\eqref{eq:estandar-hm}.
\begin{equation}\label{eq:estandar-hm}
    \{u_1,\dots, u_m,u_{m+1}, \dots u_{2m}, z\}, \text{ with nonzero brackets }
    [u_i,u_{m+i}]=z.
\end{equation}
The Heisenberg algebra of dimension $2m+1$ will be denoted as $\fh_{2m+1}$.

The next result condenses and expands the structural algebraic properties of the oscillator $\mathbb{R}$-algebras (Propositions II.3.11 and II.3.12 of \cite{hilgert1989lie}) to any field of characteristic zero.
\begin{proposition}\label{prop:pattern-oscillator}
    The generalized $\ku$-oscillator algebra $\fd_{2m+2}=\fd(V_{2m},\varphi, \delta)$ is a solvable metric algebra under the invariant bilinear form $\varphi_\delta$ described in equation~\eqref{eq:metric-oscillator}. The nilradical $\fn(\fd_{2m+2})=\fd_{2m+2}^2=V_{2m}\oplus \ku\cdot \delta^*$ is its only maximal ideal. In particular, $\fd_{2m+2}$ is a local and indecomposable metric algebra, and its centre,
    \begin{equation*}
    Z(\fd_{2m+2})=\fd_{2m+2}^{(2)}=\ku\cdot \delta^*=Z(\fn(\fd_{2m+2})),
    \end{equation*}
    is the orthogonal subspace of $\fn(\fd_{2m+2})$ and the only minimal ideal. Moreover $\fn(\fd_{2m+2})$ is a Lie algebra of Heisenberg type in which the product is completely determine by the automorphism $\delta$ through the formulas $[u,v]=\varphi(\delta(u), v)\delta^*$ and $[\delta,u]=\delta(u)$ for all $u,v\in V_{2m}$.
    \begin{equation}\label{eq:metric-oscillator}
        \varphi_\delta (t\delta+u+s\delta^*,t'\delta+v+s'\delta^*)=ts'+t's+\varphi(u,v).
    \end{equation}
\end{proposition}
\begin{proof}
    Since $\delta$ is a linear automorphism, following \cite[Lemma 2.8]{benito2023equivalent}), from expression \eqref{eq:oscillator-bracket} we get  $\fd_{2m+2}^2=\text{Im}\, \delta+\spa \langle \varphi(\delta(u),v)\delta^*:u,v\in V_{2m}\rangle=V_{2m}\oplus \ku \cdot \delta^*$ and $Z(\fd_{2m+2})=(Z(V_{2m})\cap \text{Ker }\delta)\oplus \ku \cdot \delta^*=\ku \cdot \delta^*$. Now $\fd_{2m+2}^{(3)}=[\fd_{2m+2}^{(2)},\fd_{2m+2}^{(2)}]=0$, so $\fd_{2m+2}$ is a solvable algebra, that is, $\fr(\fd_{2m+2})=\fd_{2m+2}$. And it is not  nilpotent because of $\fd_{2m+2}^3=\fd_{2m+2}^2$. Then its Jacobson radical, $\mathcal{J}(\fd_{2m+2})=[\fd_{2m+2}, \fr(\fd_{2m+2})]=\fd_{2m+2}^2$, is just its derived algebra (see \cite[Chapter III, Section 9]{jacobson1979lie} and \cite{marshall1967frattini}). As $\mathcal{J}(\fd_{2m+2})\subseteq \fn(\fd_{2m+2})\neq \fd_{2m+2}$ and it is the intersection of the whole set of maximal ideals, it is the only maximal ideal and $\mathcal{J}(\fd_{2m+2})=\fn(\fd_{2m+2})$. The statement on the Lie bracket of two elements of $V_{2m}$ follows from equation~\eqref{eq:oscillator-bracket}. To finish the proof, we use Proposition \ref{prop:ideals} and the definition of a generalised Heisenberg algebra.
\end{proof}

Following Proposition \ref{prop:pattern-oscillator}, $Z(\fd_{2m+2})=\ku \cdot \delta^*$, $\fn(\fd_{2m+2})=\fd_{2m+2}^2$ and
\begin{equation*}
    \fd_{2m+2}(V_{2m},\varphi, \delta)=\ku \cdot\delta\oplus V_{2m}\oplus \ku \cdot\delta^*=\ku \cdot\delta\oplus\underbrace{\fn(\fd_{2m+2})}_{\fh_{2m+1}}. 
\end{equation*}
 Setting  $d=\ad_{\fd_{2m+2}} \delta$, we have $d|_{V_{2m}}=\delta$ and $d(\delta^*)=0$. Then, the nilradical is $d$-invariant and any oscillator algebra can be viewed as the split extension of an algebra of Heisenberg type $\fh_{2m+1}$ by a map $d\in \der \fh_{2m+1}$ such that $\ker d=Z(\fd_{2m+2})=Z(\fh_{2m+1})$ and 
\begin{equation*}
    \fh_{2m+1}=\im d \oplus \ker d.
\end{equation*}

\begin{proposition}\label{prop:der-oscillator}
   Any oscillator $\ku$-algebra can be obtained as a split extension of a Lie algebra of Heisenberg type $\fh_{2m+1}$ and a map $d\in \der \fh_{2m+1}$ such that $\fh_{2m+1}=\im d \oplus \ker d$ where $\ker d=Z(\fh_{2m+1})$ and the invariant vector space $\im d $ is endowed with a symmetric and nondegenerate bilinear form $\varphi$ for which $d|_{\im d}$ is $\varphi$-skew. Moreover, a self-linear map $D$ of the oscillator $\ku$-algebra  $\fd(V_{2m}, \varphi, \delta)$ is a derivation if and only if:
    \begin{itemize}
        \item [(a)] $\fn(\fd(V_{2m}, \varphi, \delta))$ is $D$-invariant and $D|_{\fn(\fd(V_{2m}, \varphi, \delta))}\in\der \fh_{2m+1}$. 
        \item [(b)] $D(\delta^*)=\alpha \delta^*$ for some $\alpha\in \ku$. 
        \item [(c)] $D(\delta(a))=[D(\delta),a]+[\delta,D(a)]$ for all $a\in V_{2m}$.
    \end{itemize}
\end{proposition}

\begin{proof}
    The first part follows from previous discussion. The second part is based on the fact that the nilradical and the centre of any Lie algebra are characteristic ideals, which means that they are invariant through derivations. So (a), (b) and (c) are necessary conditions if $D$ is a derivation. For the converse, it is straightforward to check that if $D$ is a self-linear map of the oscillator algebra satisfying (a), (b) and (c), $D$ also satisfies the identity $d([x,y])=[d(x),y]+[x,d(y)]$.
\end{proof}

    
\section{Double extensions of Oscillator $\mathbb{R}$-algebras}

    Following Theorem \ref{thm:doubleExtension}, the knowledge of the set $\der_\varphi \fd(V_{2m}, \varphi, \delta)$ allows to expand oscillator algebras to other metric algebras. The latter set is closely related to $\der \fh_{2m+1}$ which is well-known set easy to describe. Using Proposition \ref{prop:der-oscillator}, in this section we will give an explicit description of the whole sets of derivations and skew-derivations of real oscillator algebras. And we will also describe their invariant forms. The results let us obtain two countable series of mixed metric Lie algebras based on oscillator algebras.

    According to \cite{benito2013levi}, the $2$-graded decomposition $\fh_{2m+1}=V\oplus Z(\fh_{2m+1})$, where $V$ is an arbitrary $\ku$-complement, induces a natural grading on $\mathrm{End}\,(\fh_{2m+1})$ and lets us describe the derivations of $\fh_{2m+1}$ as (for a matrix description see \cite{rubin1993solvable}):
    \begin{multline*}
    \der \fh_{2m+1}=\underbrace{\{\delta: \delta\mid_V \in \mathfrak{sp}(V,b_z), \delta(Z(\fh_{2m+1})=0\}}_{\fs\, \cong\, \mathfrak{sp}_{2m}(\ku)}\oplus \\ \overbrace{\ku \cdot \widehat{id}\oplus \{\delta: \delta(V) \subseteq Z(\fh_{2m+1}), \delta(Z(\fh_{2m+1}))=0\}}^{\mathfrak{r}(\der \fh_{2m+1})}
    \end{multline*}
Here $\mathfrak{sp}(V,b_z)$ is the set of $b_z$-skew linear maps of the vector space $V$. And $\widehat{id}$ means $\widehat{id}|_V=id_V$ and $\widehat{id}|_{Z(\fh_{2m+1})}=2id_{Z(\fh_{2n+1})}$. So, the Levi subalgebra $\fs$ of $\der \fh_{2m+1}$ is a simple symplectic Lie algebra. And its solvable radical is a $(2m+1)$-dimensional Lie algebra with abelian nilpotent radical $\mathfrak{n}(\der \fh_{2m+1})=\{\delta: \delta(V) \subseteq \ku \cdot z, \delta(z)=0\}$. In matrix form the general shape of a derivation in an ordered standard basis as described in \eqref{eq:estandar-hm} is 

\begin{equation}\label{eq:der-heisenberg}
 \left(
		\begin{array}{c|c|c}
				
				 M+\alpha I_m& P & \mathbf{0}\\\hline
				Q & -M^t+\alpha I_m& \mathbf{0}\\\hline
				\mathbf{c}_1^t & \mathbf{c}_2^t & 2\alpha\\
		\end{array}
	\right),   
\end{equation}
where $\alpha\in \ku$, $\mathbf{c}_i$ are column matrices, $M,P$ and $Q$ are $m\times m$ matrices $P^t=P$ and $Q^t=Q$.

Returning to the real field and 
applying the spectral theorem, for any real $\varphi$-skew and invertible map of an euclidean space, $\delta\colon(V_{2m},\varphi)\to (V_{2m}, \varphi)$, there is an orthonormal basis $\{e_1, \dots, e_{2m}\}$ such that $\delta(e_{2i-1})=\lambda_ie_{2i}$ and $\delta(e_{2i})=-\lambda_i e_{2i-1}$ and $\lambda_1, \dots, \lambda_{2m}$ are positive real numbers (w.l.o.g. we can assume $\lambda_i\leq \lambda_{i+1}$). So any oscillator $\mathbb{R}$-algebra of dimension $2m+2$ is determined by an $m$-fold $\lambda=(\lambda_1, \dots, \lambda_m)$ of positive scalars such that $0<\lambda_1 \leq \dots \leq \lambda_m$ (to shorten $\fd_{2m+2}^\lambda(\mathbb{R})$ for a fixed $m$-fold $\lambda$),
    \begin{equation*}
        \fd_{2m+2}(\lambda_1, \dots, \lambda_m)=\mathbb{R} \cdot \delta_\lambda \oplus V_{2m}\oplus \delta_\lambda^*.
    \end{equation*}

 Applying Proposition \ref{prop:pattern-oscillator}, the structure constants respect to the basis $\delta_\lambda, e_1, \dots, e_{2m}, \delta^*_\lambda$ are determined by the entries of $\lambda$. Using $\varphi (e_i,e_j)=\delta_{ij}$ and $[u,v]=\varphi(\delta_\lambda(u), v)\delta^*_\lambda$ for all $u,v\in V$, we have

 \begin{equation}\label{eq:bracket-real-oscillator}
	\begin{cases}
		[e_{2i-1},e_{2i}]= -[e_{2i},e_{2i-1}]=\lambda_i \delta^*_\lambda,\\
		[e_p,e_q]= 0 \text{ if } (p,q)\neq(2i-i,2i),(2i, 2i-1),\\
		[\delta_\lambda, e_{2i-1}]=-[e_{2i-1},\delta_\lambda]=\delta_\lambda(e_{2i-1})=\lambda_ie_{2i},\\
        [\delta_\lambda, e_{2i}]=-[e_{2i},\delta_\lambda]=\delta_\lambda( e_{2i})=-\lambda_ie_{2i-1},\\
        [\fd_{2m+2},\delta^*_\lambda] = 0.
	\end{cases}
\end{equation}
From the basis-bracket description of $\fd_{2m+2}^\lambda(\mathbb{R})$, the next lemma restates \cite[Proposition II.3.14]{hilgert1989lie}.

\begin{lemma}
 Let $\lambda=(\lambda_1,\dots,\lambda_m)$ and $\fd_{2m+2}^\lambda(\mathbb{R})=\mathbb{R}\cdot \delta_\lambda\oplus \spa\langle e_1, \dots, e_{2m}\rangle\oplus \mathbb{R}\cdot \delta_\lambda^*$ the oscillator algebra with Lie bracket given in equations~\eqref{eq:bracket-real-oscillator}. For any $t \in \ku$ and $s\neq 0$ the symmetric bilinear form $\varphi_{t,s}$ given as the orthogonal sum $\spa \langle e_1, \dots, e_{2m} \rangle\perp \spa \langle \delta_\lambda,\delta_\lambda^*\rangle$, $\varphi_{t,s}(\delta_\lambda,\delta_\lambda)=t$, $\delta_\lambda$ and $\delta_\lambda^*$ isotropic, $\{e_1, \dots, e_{2m}\}$ orthogonal family and $\varphi_{t,s}(\delta_\lambda,\delta_\lambda^*)=s=\varphi_{t,s}(e_i,e_i)$ is invariant and nondegenerate. The set $\{\varphi_{t,s}: t, s \in \ku, s\neq 0 \}$ is the whole set of symmetric invariant and nondegenerate bilinear forms of $\fd_{2m+2}^\lambda$. In particular, $B_{inv}(\fd_{2m+2}^\lambda)=\spa \langle \varphi_{0,1}, \varphi_{1,1}\rangle$ and the metric dimension of real oscillator algebras is two.
\end{lemma}
\begin{proof}
    Note that $\determ \varphi_{t,s}=s^{2m+2}$, so  $\varphi_{t,s}$ is nondegenerate if and only if $s\neq 0$. The invariance of $\varphi_{t,s}$ is equivalently to 
    \begin{equation*}
        \varphi_{t,s}([x,a],b)+\varphi_{t,s}(a,[x,b])=0 \quad \forall x,a,b\in \fd_{2m+2}.
    \end{equation*}
    The equality follows by checking it for $x\in \{\delta_\lambda, e_i\}$ (only  $\varphi_{0,1}$ and $\varphi_{1,1}$ need to be checked). Now let $b$ an arbitrary invariant symmetric and nondegenerate form. From Proposition \ref{prop:ideals} $(\fd_{2m+2}^2)^\perp=Z(\fd_{2m+2})=\ku \delta_\lambda^*$ and  $b(\delta_\lambda,\delta_\lambda^*)=s_0 \neq 0$ because $b$ is non-degenerate. We also set $b(\delta_\lambda,\delta_\lambda)=t_0$. Since $\delta_\lambda(e_{2i-1})=\lambda_ie_{2i}$ and $\delta_\lambda(e_{2i})=-\lambda_ie_{2i-1}$, we get
    \begin{equation*}
        b(\delta_\lambda, e_{2i})=\frac{1}{\lambda_i}b(\delta_\lambda, \delta_\lambda(e_{2i-1}))=b(\delta_\lambda, [\delta_\lambda,e_{2i-1}])=0
    \end{equation*}
    by using that $b$ is invariant. So $b(\delta_\lambda, e_{2i})=0$ and $b(\delta_\lambda, e_{2i-1})=0$ in the same vein. Finally, from \eqref{eq:bracket-real-oscillator}, $b(\delta_\lambda, [e_{2i-1},e_{2j}])=\delta_{ij}\lambda_i s_0$ and by invariance: 
    \begin{equation*}
	\begin{cases}
		b(\delta_\lambda, [e_{2i-1},e_{2j}])=b(\delta_\lambda(e_{2i-1}),e_{2j})=\lambda_ib(e_{2i}, e_{2j}),\\
		b(\delta_\lambda, [e_{2i-1},e_{2j}])=-b(\delta_\lambda(e_{2j}),e_{2i-1})=\lambda_jb(e_{2j-1}, e_{2i-1}).
	\end{cases}
\end{equation*}
    Therefore, $b(e_{2i}, e_{2j})=b(e_{2i-1}, e_{2j-1})=\delta_{ij}s_0$. A similar reasoning yields
    \begin{equation*}
        b(e_{2i}, e_{2j-1})=\frac{1}{\lambda_i}b(\delta_\lambda(e_{2i-1}), e_{2j-1})={\lambda_i}b(\delta_\lambda,[e_{2i-1}), e_{2j-1}])=0,
    \end{equation*}
    so $b=\varphi_{t_0,s_0}$. The final assertion follows from Proposition~\ref{prop:inv-sym-mdim} and the linearly depending relation $\varphi_{t_0,s_0}=(s_0-t_0)\varphi_{0,1}+ t_0\varphi_{1,1}$.
\end{proof}

There are three classes of metric real oscillator algebras depending on $\lambda=(\lambda_1, \dots, \lambda_m)$ (see \cite[Section 4]{medina1985algebres}):
\begin{itemize}
    \item [] $\mathcal{O}$-I: all the entries of $\lambda$ are different. Then, the set of $\varphi_{0,1}$-skew derivations is an abelian Lie algebra.
    
    \item [] $\mathcal{O}$-II: all the entries of $\lambda$ are equals to $\lambda_1$. Up to isomorphisms for any $m\geq 1$ we have the series $\fd_{2m+2}(1,\dots, 1)$. Since  $\lambda=(\lambda_1, \dots, \lambda_1)$, rescaling the basis of $\fd_{2m+2}^\lambda(\mathbb{R})$ in the form $\frac{1}{\lambda_1}\delta, e_1,\dots, e_{2m}, \lambda_1 \delta_\lambda^*$  we arrive at $\fd_{2m+2}(1,\dots, 1)$. The set of $\varphi_{0,1}$-skew derivations is the special unitary Lie algebra $\fsu_m(\mathbb{R})$, a simple Lie algebra of type $A$ (i.e. the complex extension $\fsu_m(\mathbb{R})\otimes \mathbb{C}$ is $\fsl_m(\mathbb{C})$).
    \item [] $\mathcal{O}$-III: there are at least two different entries $\lambda_i<\lambda_{i+k}$ and one of them of multiplicity $\geq 2$. The set of $\varphi_{0,1}$-skew derivations is a reductive non-abelian Lie algebra. 
\end{itemize}
\noindent To end this section, we will compute explicitely the whole sets of derivations and $\varphi_{0,1}$-skew derivations of  $\fd_{2m+2}(1,\dots, 1)$ (for short $\fd_{2m+2}$). In the sequel, we fixed a natural $m\geq 1$ and, in order to get a more symmetric block description of any derivation, all the self-linear maps of $\fd_{2m+2}$ will be given in a matrix level with respect to the ordered basis \[\{\delta_\lambda, x_1, \dots, x_m, y_1, \dots,y_m, \delta_\lambda^*=z\}\] with $x_i=e_{2i-1}$ and $y_i=e_{2i}$. In this way, $\{x_i,y_i, \delta_\lambda^*\}$ forms a standard basis of $\fh_{2m+1}$ as in expression~\eqref{eq:estandar-hm}. For the rest of products we observe equations~\eqref{eq:bracket-real-oscillator} with $\lambda_i=1$. Let $D$ be any derivation of $\fd_{2m+2}$. From 
\begin{equation*}
    D(\delta_\lambda)=\gamma \delta_\lambda+ \sum_{i=1}^m b_i x_i+ \sum_{i=1}^m c_i y_i +\beta z,
\end{equation*}
items (a), (b) and (c) in Proposition \ref{prop:der-oscillator} and matrix description in \eqref{eq:der-heisenberg}, we arrive at the general matrix description:

\begin{equation}\label{eq:der-diamond-general}
 D=\left(
		\begin{array}{c|cc|c}
				0&\multicolumn{1}{c|}{\mathbf{0}} & \mathbf{0}& 0 \\\hline
				\multicolumn{1}{c|}{\mathbf{b}} & \multicolumn{1}{c|}{M+\alpha I_m} & P & \mathbf{0}\\\hline
				\mathbf{c} & \multicolumn{1}{c|}{-P} & -M^t+\alpha I_m& \mathbf{0}\\\hline
				\beta & \multicolumn{1}{c|}{-\mathbf{b}^t} & -\mathbf{c^t}&2\alpha\\
		\end{array}
	\right),   
\end{equation}
with $\beta, \alpha \in \mathbb{R}$, $\mathbf{b}, \mathbf{c}$ $1\times m$ matrices and $M^t=-M$ and $P^t=P$.
The set of inner derivations $\inner \fd_{2m+2}=\spa \langle \ad \delta_\lambda, \ad x_i, \ad y_i : 1\leq i\leq m\rangle$ is just the set of matrices as in equation~\eqref{eq:der-diamond-general} with $M=0$, $\beta=\alpha=0$ and $P=\mu I_m$. Using $\ad[x,y]=[\ad x, \ad y]$, we get the derived subalgebra of this algebra, $(\inner \fd_{2m+2})^2=\spa \langle \ad x_i, \ad y_i : 1\leq i\leq m\rangle$, which is clearly abelian.

Among the derivations of $(\fd_{2m+2}, \varphi_{0,1})$, we look for the $\varphi_{0,1}$-skew ones:
\begin{equation}\label{eq:skew-der-diamond}
    D\in \der_{\varphi_{0,1}} \fd_{2m+2}\Longleftrightarrow \varphi_{0,1}(D(x),y)+ \varphi_{0,1}(x,D(y))=0.
\end{equation}
\noindent From \eqref{eq:skew-der-diamond}, the skew-derivations are as in \eqref{eq:der-diamond-general} with $\beta=\alpha=0$. Hence, any derivation $D$ decomposes into the generic sum of basic blocks of derivations:
\begin{multline*}
\footnotesize{
    D=\overbrace{\left(
			\begin{array}{c|cc|c}
				0&\multicolumn{1}{c|}{\mathbf{0}} & \mathbf{0}& 0 \\\hline
				\multicolumn{1}{c|}{\mathbf{0}} & \multicolumn{1}{c|}{\alpha I_m} & \mathbf{0} & \mathbf{0}\\\hline
				\mathbf{0} & \multicolumn{1}{c|}{\mathbf{0}} & {\alpha I_m}& \mathbf{0}\\\hline
				0 & \multicolumn{1}{c|}{\mathbf{0}} & \mathbf{0}& 2\alpha\\
		\end{array}
		\right)}^{\displaystyle\alpha D_{0,1,2}}+
		\underbrace{\overbrace{\left(
		\begin{array}{c|cc|c}
				0&\multicolumn{1}{c|}{\mathbf{0}} & \mathbf{0}& 0 \\\hline
				\multicolumn{1}{c|}{\mathbf{0}} & \multicolumn{1}{c|}{M} & \mathbf{0} & \mathbf{0}\\\hline
				\mathbf{0} & \multicolumn{1}{c|}{\mathbf{0}} & M& \mathbf{0}\\\hline
				0 & \multicolumn{1}{c|}{\mathbf{0}} & \mathbf{0}&0\\
		\end{array}
		\right)}^{\displaystyle\fs_0}+
		\overbrace{\left(
		\begin{array}{c|cc|c}
				0&\multicolumn{1}{c|}{\mathbf{0}} & \mathbf{0}& 0 \\\hline
				\multicolumn{1}{c|}{\mathbf{0}} & \multicolumn{1}{c|}{\mathbf{0}} & P_0 & \mathbf{0}\\\hline
				\mathbf{0} & \multicolumn{1}{c|}{-P_0} & \mathbf{0} & \mathbf{0}\\\hline
				0 & \multicolumn{1}{c|}{\mathbf{0}} & \mathbf{0}&0\\
		\end{array}
		\right)}^{\displaystyle \in\fs_1}
	     }_{\displaystyle \fs^2=[\fs,\fs]_\fs}+
	     }\\
\footnotesize{
		\underbrace{\overbrace{\left(
			\begin{array}{c|cc|c}
				0&\multicolumn{1}{c|}{\mathbf{0}} & \mathbf{0}& 0 \\\hline
				\multicolumn{1}{c|}{\mathbf{0}} & \multicolumn{1}{c|}{\mathbf{0}} & \mu I_m & \mathbf{0}\\\hline
				\mathbf{0} & \multicolumn{1}{c|}{-\mu I_m} & \mathbf{0}& \mathbf{0}\\\hline
				0 & \multicolumn{1}{c|}{\mathbf{0}} & \mathbf{0}& 0\\
		\end{array}
		\right)}^{\displaystyle \mu\ad \delta_\lambda\in \fs_1}+
		\overbrace{\left(
			\begin{array}{c|cc|c}
				0&\multicolumn{1}{c|}{\mathbf{0}} & \mathbf{0}& 0 \\\hline
				\multicolumn{1}{c|}{\mathbf{b_1}} & \multicolumn{1}{c|}{\mathbf{0}} & \mathbf{0} & \mathbf{0}\\\hline
				\mathbf{b_2} & \multicolumn{1}{c|}{\mathbf{0}} & \mathbf{0}& \mathbf{0}\\\hline
				0 & \multicolumn{1}{c|}{\mathbf{-b_1}^t} & \mathbf{-b_2^t}& 0\\
		\end{array}
		\right)
		}^{\displaystyle\mathfrak{t}=(\inner \fd_{2m+2})^2}
		}_{\displaystyle\inner \fd_{2m+2}}+
		\overbrace{\left(
			\begin{array}{c|cc|c}
				0&\multicolumn{1}{c|}{\mathbf{0}} & \mathbf{0}& 0 \\\hline
				\multicolumn{1}{c|}{\mathbf{0}} & \multicolumn{1}{c|}{\mathbf{0}} & \mathbf{0} & \mathbf{0}\\\hline
				\mathbf{0} & \multicolumn{1}{c|}{\mathbf{0}} & \mathbf{\mathbf{0}}& \mathbf{0}\\\hline
				\beta & \multicolumn{1}{c|}{\mathbf{0}} & \mathbf{0}& 0\\
		\end{array}
		\right)}^{\displaystyle\beta D_{1,0,0}},%
	     }
\end{multline*}
where $M^t=-M$ and $P_0^t=P_0$ are $m\times m$ traceless matrices. Let denote $\fs$ the set of $2m\times 2m$ matrices of the following shape:
\begin{equation}\label{eq:jordan-triple}
        \left(\begin{array}{c|c} M& P\\ \hline -P& M \\\end{array}\right)=\left(\begin{array}{c|c} M& \mathbf{0}\\ \hline \mathbf{0}& M \\\end{array}\right)\oplus\left(\begin{array}{c|c} \mathbf{0}& P_0\\ \hline -P_0& \mathbf{0} \\\end{array}\right)\oplus\left(\begin{array}{c|c} \mathbf{0}& \mu I_m\\ \hline -\mu I_m& \mathbf{0} \\\end{array}\right),
    \end{equation}
here $P=P_0+\mu I_m$ and $\mu=\frac{\tr P}{m}$. It is easily checked that $\fs$ is a vector space  which is closed under the bracket $[x,y]_{\fs}=xy-yx$. Then, $\fs$ is a linear Lie subalgebra of the special linear algebra $\fsl_{2m}(\mathbb{R})$ and the direct sum decomposition given in \eqref{eq:jordan-triple} provides a $\mathbb{Z}_2$-graded decomposition $\fs=\fs_0\oplus \fs_1$. So, the even part $\fs_0$ is a Lie algebra, in this case isomorphic to the simple orthogonal algebra of skew symmetric matrices $\fso_m(\mathbb{R})$ if $m\geq 3$.

\begin{theorem}
Let $\fd_{2m+2}^\lambda(\mathbb{R})=\mathbb{R} \cdot \delta_\lambda \oplus (V_{2m}, \varphi_{0,1})\oplus \mathbb{R} \cdot \delta_\lambda^*$, $m\geq 2$, be the real oscillator Lie algebra of $m$-fold $\lambda=(1,\dots, 1)$. The sets of derivations and skew-symmetric derivations can be described as follows:
    \begin{itemize}
        \item [(a)] $\der \fd_{2m+2}=\mathbb{R} \cdot D_{1,0,0} \oplus\mathbb{R} \cdot  D_{0,1,2}\oplus [\fs,\fs]_\fs\oplus \inner \fd_{2m+2}$ where $D_{0,1,2}$ is the derivation given by $D_{0,1,2}(\delta_\lambda)=0$, $D_{0,1,2}(v)=v$ for all $v\in V$ and $D_{0,1,2}(z)=2z$ and $D_{1,0,0}(\delta_\lambda)=z$ and $D_{1,0,0}(\fd_{2m+2}^2)=0$.
        
        \item [(b)] $\der_{\varphi_{0,1}} \fd_{2m+2}=\mathfrak{s}\oplus [\inner \fd_{2m+2},\inner \fd_{2m+2}]$, and $\mathfrak{s}^2$ is isomorphic to the special unitary simple Lie algebra $\fsu_m(\mathbb{R})$. For $m\geq 3$, $\fs_0$ is the orthogonal simple algebra of $m\times m$ skew-matrices.
        \end{itemize}
        For $m=1$, $\der \fd_4=\mathbb{R} \cdot D_{0,1,2}\oplus\mathbb{R} \cdot D_{1,0,0} \oplus \inner \fd_4$ and $\der_{\varphi_{0,1}} \fd_4=\inner \fd_{4}$.
\end{theorem}

\begin{proof}
  The result follows from previous matrix decompositions and discussion. Since the special unitary real Lie algebra can be realized as the vector space of traceless skew-Hermitinan $m\times m$ matrices, so $\fsu_m(\mathbb{R})=\{M+iP:M^t=-M, P^t=P, \mathrm{tr}\, P=0\}$ and $\fso_m(\mathbb{R})=\{M:M^t=-M\}$ is a subalgebra. It is easily checked that the map $M+P_0\mapsto M-iP_0$ (here $M+P_0$ represents the two first summands in the decomposition \eqref{eq:jordan-triple}) is a Lie isomorphism from $\fs^2$ to $\fsu_m(\mathbb{R})$. The same map  proves that $\fs_0\cong \fso_m(\mathbb{R})$.    
 \end{proof}

\begin{remark}
Let $J=H(M_m(\mathbb{R},t))$ be the real simple unitary Jordan algebra of $m\times m$ symmetric matrices for $m\geq 2$. Up to isomorphisms, $\fs_0$ and $\fs=\mathbb{R} \cdot \ad \delta_\lambda\oplus [\fs,\fs]_\fs$ are just the algebra of derivations of $J$ and  the Lie multiplication algebra of $J$ according to \cite[Chapter VI, Section 9, Theorems 9 and 11]{jacobson1968structure}. And the $\mathbb{Z}_2$-graded decomposition $\fs^2=\fs_0\oplus \fs_1\cap \fs^2$ is related to the compact symmetric space $SU(m)/SO(m)$ (see \cite[Table V, page 518]{helgason1979differential}). We also point out that Jordan algebras were introduced by Pascual Jordan in 1933 to formalize the notion of an algebra of observables in quantum mechanics. The algebra $\fd_4(\mathbb{R})$ is the algebra of the observables of the quantum mechanical model of the harmonic oscillator.
\end{remark}
The existence of $\fs^2\cong \fsu_m(\mathbb{R})$ and $\fs_0\cong \fso_m(\mathbb{R})$ as simple subalgebras of $\der_{\varphi_{0,1}} \fd_{2m+2}$ for $m\geq 2$ and $m\geq 3$ lets us construct, in parallel with the quadratic solvable series $(\fd_{2m+2}, \varphi_{0,1})$, the series of mixed quadratic algebras,
    \begin{equation*}
        (\fd_{2m+2})_{\fsu_m(\mathbb{R})}:= \fsu_m(\mathbb{R}) \oplus \fd_{2m+2} \oplus \fsu_m(\mathbb{R})^*, \text{ and }
    \end{equation*}
    \begin{equation*}
        (\fd_{2m+2})_{\fso_m(\mathbb{R})}:= \fso_m(\mathbb{R}) \oplus \fd_{2m+2} \oplus \fso_m(\mathbb{R})^*,
    \end{equation*}
by following Theorem \ref{thm:doubleExtension}. In this case, we extend from the natural inclusion of $\fs$ into the set of skew-derivations, so $\fb=\fs^2, \fs_0$ and $\phi=\iota$. 

Since $\der_{\varphi_{0,1}} \fd_4=\inner \fd_{4}$, double extensions of the oscillator algebra $\fd_4(\mathbb{R})$ by means of its skew-derivations only produce decomposable metric algebras that are orthogonal sums $\fd_4\oplus \fa$,  with $\fa$  metric abelian (see  \cite[Lemma 2.8]{benito2023equivalent}).

\section*{Funding}
The authors have been supported by research grant MTM2017-83506-C2-1-P of `Ministerio de Econom\'ia, Industria y Competitividad, Gobierno de Espa\~na' (Spain) until 2022 and by grant PID2021-123461NB-C21, funded by MCIN/AEI/10.13039/501100011033 and by “ERDF A way of making Europe” since then. J. Rold\'an-L\'opez was also supported by a predoctoral research grant FPI-2018 of `Universidad de La Rioja'.

\bibliographystyle{apalike}
\bibliography{bibliography}

\end{document}